\documentclass[a4paper,12pt,draft]{amsart}
\usepackage{amsmath}
\usepackage{amsfonts}
\usepackage{amssymb,amsthm,amscd}
\usepackage[OT4]{fontenc}
\usepackage{fullpage}
\newtheorem{theorem}{Theorem}
\newtheorem{lemma}[theorem]{Lemma}
\newtheorem{fact}[theorem]{Fact}
\newtheorem{corollary}[theorem]{Corollary}
\newtheorem{observation}[theorem]{Remark}
\newtheorem{proposition}[theorem]{Proposition}
\newtheorem{definition}[theorem]{Definition}

\numberwithin{theorem}{section}
\numberwithin{equation}{section}
\DeclareMathOperator\dc{{\it dd^c}}

\DeclareMathOperator\we{\wedge}
\DeclareMathOperator\om{{\it\omega}}
\DeclareMathOperator\de{{\it\delta}}
\DeclareMathOperator\et{{\it\eta}}
\DeclareMathOperator\ep{{\it\epsilon}}
\DeclareMathOperator\Om{{\it\Omega}}

\DeclareMathOperator \caw {{\it cap_{\om}}}
\DeclareMathOperator \wph {{\it\om_{u}}}
\DeclareMathOperator \wphj {{\it\om_{u_j}}}
\DeclareMathOperator \wps {{\it\om_{v}}}
\DeclareMathOperator \wphn {{\it\om_{u}^n}}

\DeclareMathOperator \wpsn {{\it\om_{v}^n}}

\DeclareMathOperator \wpsh {{\it PSH(X,\om)}}

\DeclareMathOperator \lbr {{\it\lbrace}}
\DeclareMathOperator \rbr {{\it\rbrace}}
\DeclareMathOperator \eu {{\it e^{\phi}}}

\DeclareMathOperator \dd {\partial }
\DeclareMathOperator \inpp{{\it \int_{\lbrace u< v\rbrace}}}
\begin{document}
\title{Pluripotential estimates on compact Hermitian manifolds}
\author{S\l awomir Dinew, S\l awomir Ko\l odziej}

\maketitle

\begin{abstract}We discuss pluripotential aspects of the Monge-Amp\`ere equations on compact Hermitian manifolds and prove $L^{\infty}$ estimates for any metric, as well as the existence of weak solutions under an extra assumption.
\end{abstract}
\section{introduction}
The complex Monge-Amp\`ere equation is a fundamental tool in complex geometry. It appears in the problem of prescribing the Ricci curvature on compact K\"ahler manifold in a fixed K\"ahler class. The solution of Calabi conjecture by Shing-Tung Yau in 1976 (\cite{Y}), which  says that the complex Monge-Amp\`ere equation can always be solved for any smooth volume form satisfying the necessary normalization condition revolutionized complex differential geometry and is still an object of inspiration for future studies. Let us mention that the parabolic version of the (elliptic) Monge-Amp\`ere equation, namely the K\"ahler-Ricci flow is one of the central themes  in modern differential geometry. From the point of view of complex dynamics, in turn, the Monge-Amp\`ere equation produces potentials for singular measures and it is a central problem to relate the regularity of such potentials to the regularity properties of the measure itself (see, for example \cite{DNS}).

There are many modifications and generalizations in the existing literature. In particular, one can study the Monge-Amp\`ere equation with respect to a general Hermitian metric instead of a K\"ahler one. In such a case the equation is not so geometric, since Hermitian metrics do not represent positive cohomology classes (see however \cite{TW1} for some geometrical applications). On the other hand any compact complex manifold admits a Hermitian metric which is not the case for K\"ahler metrics. Thus for example Hopf manifolds are beyond the scope of the K\"ahler theory. Here we would like to mention the recent preprint of Streets and Tian \cite{StT}, where a Hermitian version of the K\"ahler-Ricci flow is being considered, aiming at important geometrical applications. The study of non-K\"ahler metrics is also motivated by physics as it is, for instance, in Fu and Yau paper \cite{FY}.

In complex dynamics positive non-closed currents are also studied. One can mention here the concept of a {\it harmonic current} which is not necessarily $d$  closed but $\dc$ closed. Laminations associated to such currents are studied for instance in \cite{FS}.

In the eighties and nineties some results regarding the Monge-Amp\`ere equation in the Hermitian setting were obtained by Cherrier and  Hanani (\cite{Ch}, \cite{CH1}, \cite{CH2}, \cite{H1}, \cite{H2}). For next few years there seems to be no activity on the subject until very recently, when the results were rediscovered and generalized by Guan-Li (\cite{GL},\cite{GL2}) and Tosatti-Weinkove (\cite{TW1}, \cite{TW2}). The PDE techniques applied in all these papers are similar to the a priori estimates due to S. T. Yau from the K\"ahler case (\cite{Y}). Therefore a natural question (posed in \cite{GL}, \cite{GL2} and \cite{TW1}) arises whether one can obtain $L^{\infty}$ estimates via suitably constructed pluripotential theory in this setting.

Indeed, there are obstructions for such a theory. First of all in the non-K\"ahler case there are no local potentials for the metric and hence the methods from \cite{koj98} as well as \cite{Bl1} do not directly apply. However it is still possible to work globally as in \cite{koj03}, that is independently of the local potentials. In such an approach the main difficulty that arises in transplanting the theory directly to Hermitian setting is the comparison principle which says the following:

{\it Let $\phi,\ \psi \in \wpsh\cap L^{\infty}(X)$. If $\om$\ is a K\"ahler metric then
$$\inpp\wpsn\leq\inpp\wphn.$$}

The aim of this paper is to develop some basic pluripotential techniques in the Hermitian setting and thereby prove the analogues of the results from \cite{koj05}. As the theory is in many aspects similar to the  K\"ahlerian one, we stress mainly those aspects where either differences occur and/or arguments need suitable modification.

After some preliminary material we discuss different variants of the comparison principle in the Hermitian setting. In particular we show that the assumptions used in the Guan and Li paper \cite{GL} are perfectly suitable for a meaningful pluripotential theory, while the assumption of $d$-balanced metrics (from \cite{TW1}) is not well adapted and causes some additional technical difficulties.

In Section \ref{apriori} we obtain $L^{\infty}$ a priori estimates for weak solutions to the Monge-Amp\`ere equation with Monge-Amp\`ere measure dominated by capacity, generalizing thus results in \cite{koj98}, \cite{koj03} from the K\"ahler setting. In particular solutions, provided that they exist, are bounded for all measures with $L^p$ densities ($p>1$). Another corollary is that some singular measures such as volumes of real hypersurfaces may also admit bounded potentials with respect to any hermitian metric. Since the details of the proof are technically involved below we sketch the main ideas.

Essentially we follow the proof for the K\"ahler case from \cite{koj03} (see also \cite{EGZ} for a similar argument). We have to cope with the lack of a comparison principle for a general hermitian metric. The first main idea (Theorem \ref{main1}) is that a comparison principle type estimate is available provided one can control the oscillation of the difference of the  functions involved and additionally one has to allow some error terms. These error terms involve Hessian type operators of lower order (such as Laplacians with respect to the metric). Contrary to the smooth case (\cite{TW1}, \cite{GL}) one cannot control these terms pointwise. Instead, the crucial observation (Lemma \ref{main2}) is that it is enough  to work only with  sublevel sets for levels close to infimum and then control the integrals of all such error terms over those sets by the Monge-Amp\`ere mass. Coupling these two results, as well as a uniform estimate for the the capacity of sublevel sets (Proposition \ref{capest}) one obtains a weak version of the fundamental Lemma 2.2 from \cite{koj03} which is satisfactory for our needs.

In the next section we assume that a Hermitian metric $\omega$ satisfies
 \begin{equation}
\forall u\in \wpsh\cap L^{\infty}(X)\ \ \int_X(\om+\dc u)^n=\int_X\om^n.
\end{equation}
 This assumption is equivalent to validity of the usual comparison principle. We then show the existence of continuous solutions of the degenerate Monge-Amp\`ere equation when the right hand side belongs to some Orlicz space, in particular when it is in $L^p$  ($p>1$). This result for smooth data, and without the extra hypothesis, was recently obtained by Tosatti-Weinkove in \cite{TW2}, building on Cherrier's \cite{Ch} who proved it for balanced metrics.

The last section contains some remarks regarding the uniqueness of solutions. In the smooth case this was very recently established by Tosatti and Weinkove (\cite{TW2}). We propose alternative elementary proof which is in the spirit of pluripotential theory. We believe that our approach has the slight advantage that it could be more easily adapted to non smooth setting.

The tools developed in this paper make it possible to extend stability and H\"older regularity theorems from \cite{koj03} \cite{koj08} to the Hermitian case at least with the above assumption. They can also be applied for manifolds with a boundary considered in \cite{GL}. This will be done in  subsequent papers.

{\bf Notation.} As it is customary $C$ will denote different constants that may vary line-to-line. In arguments where a lot of constants appear we shall either enumerate or indexate such $C$'s in order to make a distinction between them. An exception is made for the special constant $B$ (see (\ref{B}) for its definition), since it is heavily linked to the geometry of the manifold and plays crucial role in the proofs.

{\bf Acknowledgements.} The research was done while the first named author was a Junior Research Fellow at the Erwin Schr\"odinger Institute, Vienna. He wishes to express his gratitude to this institution for the perfect working conditions as well as for the financial support.
 
 The second author partially supported by EU grant
MTKD-CT-2006-042360 and Polish   grants 189/6 PR UE/2007/7,
3679/B/H03/2007/33.
\section{preliminaries}
Throughout the note $X$ will be a fixed compact Hermitian manifold equipped with a fixed Hermitian metric $\om$ and $n=dim_{\mathbb C}X$. Also $d=\partial+\overline{\partial}$\ denotes the standard operator of exterior differentiation while $d^c:=i/2\pi(\overline{\partial}-\partial)$.

Let the universal constant $B>1$ satisfy the following inequalities:
\begin{equation}\label{B}
 -B\om ^2 \leq n\dc\om\leq B\om ^2,\ \ -B\om ^3 \leq n^2d\om\we d^c\om\leq B\om ^3 .
\end{equation}

We define the function class
\begin{equation*}
 PSH(X,\omega):=\lbrace u \in L^1(X,\omega):dd^c u\geq -\omega,\ u \in\mathcal C^{\uparrow}(X) \rbrace,
\end{equation*}
where $\mathcal C^{\uparrow}(X)$\ denotes the space of upper semicontinuous functions and the inequality is understood in the weak sense of currents. We call the functions that belong to $PSH(X,\omega)$ $\omega$-plurisubharmonic ($\om$-psh for short). We shall often use the handy notation $\wph:=\om+dd^c u$.

Contrary to the K\"ahler case the form $\om$\ need not have local potentials, nevertheless $\om$-psh  functions are locally standard plurisubharmonic functions plus some smooth function. This follows from the fact that in normal coordinates (see e.g. \cite{StT} or \cite{GL}) at a given point $\om =i\sum dz_j \we d\bar{z} _j$. We shall refer to that fact in the sequel as to  \it psh-like property  \rm of $\om$-psh  functions.
This property allows to use some local results from pluripotential theory developed by Bedford and Taylor in \cite{BT2}. In particular the Monge-Amp\`ere operator
$$\wphn:=\om_{u}\we\cdots\we\om_{u}$$
 is well defined  for bounded $\om$-psh functions.
 Furthermore, if  $u_j \in PSH(\om )\cap L^{\infty}$ is either decreasing or increasing almost everywhere to $u$, then
 $$
 (dd^c u_j +\om )^n \to (dd^c u +\om )^n \ $$
 in the sense of currents. This follows from the convergence theorems in \cite{BT2} via the following argument.
   Suppose $\om$ is a  Hermitian form in a ball $B$, and $\Om$ a K\"ahler form
 such that $ \om <\Om .$ Write
 $$dd^c u_j +\om = dd^c u_j +\Om - T, \ \  T=(\Om -\om ).$$
Then by the Newton expansion
\begin{equation}\label{de}
(dd^c u_j +\om )^n =  (dd^c u_j +\Om )^n -n (dd^c u_j +\Om )^{n-1}\we T + ... \pm T^n .
\end{equation}
By the convergence theorem for psh functions  \cite{BT2} all the terms on the right converge as currents, and the sum of their limits is
$$
(dd^c u +\Om )^n -n (dd^c u +\Om )^{n-1}\we T +... \pm T^n =(dd^c u +\om )^n .$$

We note that {\it all} functions $u$ in $\wpsh$, normalized by the condition $sup_X u=0$ are uniformly integrable. This follows from classical  results in potential theory and psh-like property as in \cite{koj98}.
Since such results seem to be important in a more general setting (compare \cite{TWY}) we give here a complete argument following quite closely the one in \cite{GZ1}, where the authors treat the K\"ahler case.
\begin{proposition}\label{l1int}
 Let $u\in\wpsh$ be a function satisfying $sup_X u=0$. Then there exists a constant $C$ dependent only on $X,\ \om$ such that
$$\int_X| u|\om^n\leq C.$$
\end{proposition}
 \begin{proof} Consider a double covering of $X$ by coordinate balls $B_{s}^1\subset\subset B_{s}^2\subset X,\ s=1,\cdots, N$. In each $B_{s}^2$ there exists a strictly plurisubharmonic potential $\rho_s$ satisfying the following properties:
\begin{equation}\label{potentials}
 \begin{cases}
  \rho_s|_{\partial B_s^2}=0\\
  inf_{B_s^2}\rho_s\geq-C_1\\
  \dc\rho_s=\om_{2,s}\geq\om,
 \end{cases}
\end{equation}
where $C_1$ is a constant dependent only on the covering and $\om$.
Suppose now that there exist a sequence $u_j\in\wpsh, sup_X u_j=0$ satisfying $lim_{j\rightarrow\infty}\int_X| u_j|\om^n=\infty$. After choosing subsequence (which for the sake of brevity we still denote by $u_j$) we may assume that
\begin{equation}\label{l1counterproof}
\int_X| u_j|\om^n\geq 2^j
\end{equation}
and moreover a sequence of points $x_j$ where $u_j$ attains maximum is contained in some fixed ball $B_{s}^1$.

Note that $\rho_s+u_j$\ is an ordinary plurisubharmonic function in $B_s^2$ and by the sub mean value property one has
\begin{equation}\label{1}
 \rho_s(x_j)=\rho_s(x_j)+u_j(x_j)\leq C_2\int_{B_s^2} \rho_s(z)+u_j(z)d\lambda\leq C_2\int_{B_s^2} u_j(z)d\lambda+C_3,
\end{equation}
where $d\lambda$ is the Lebesgue measure in the local coordinate chart, while $C_2,\ C_3$ are constants dependent only on $B_s^1$ and $B_s^2$. Thus (\ref{1}) implies that for some constant $C_4$\ one has
\begin{equation}\label{2}
\int_{B_s^2} |u_j(z)|d\lambda\leq C_4.
\end{equation}
Consider the function $v:=\sum_{j=1}^{\infty}\frac{u_j}{2^j}$. By classical potential theory this is again a $\om$-psh function or constantly $-\infty$. By (\ref{2}) however the integral of $v$ over $B_s^2$ is finite thus it is a true $\om$-psh function. Reasoning like in the fixed ball $B_s^2$ we easily obtain that $v\in L^1(B_t^1)$ for any $t\in 1,\cdots, N$ and hence $v\in L^1(X)$. This  contradicts (\ref{l1counterproof}), and thus the existence of an uniform bound is established.
 \end{proof}

Next we define a capacity in the Hermitian setting (\cite{koj03}):
\begin{definition}
 For any Borel set $E\subset X$ we define the capacity $\caw$ by the formula
$$\caw(E):=sup\lbr\int_E\wphn| u\in\wpsh,\ 0\leq u\leq1\rbr.$$
\end{definition}

Contrary to the K\"ahler case it is not obvious that this is always a finite quantity. This follows however from the next proposition.
\begin{proposition}
 For any $k\in\lbr0,\cdots,n\rbr$ the quantity
$$sup\lbr\int_X\wph^k\we\om^{n-k}| u\in\wpsh,\ 0\leq u\leq1\rbr$$
is finite.
\end{proposition}
\begin{proof}
 For $k=0$ the statement trivially holds. Assume now $k>0$ and let us fix a function $u$ which is a competitor for the supremum. Then one has the following
\begin{align*}
 &\int_X\wph^k\we\om^{n-k}=\int_X\wph^{k-1}\we\om^{n-k+1}+\int_X\dc u\we \wph^{k-1}\we\om^{n-k}\\
&=\int_X\wph^{k-1}\we\om^{n-k+1}+\int_X u\dc[\wph^{k-1}\we\om^{n-k}]=\int_X\wph^{k-1}\we\om^{n-k+1}\\
&+\int_X u[(k-1)\dc\om\we\wph^{k-2}\we\om^{n-k}-(k-1)(k-2)d\om\we d^c\om\we\wph^{k-3}\we\om^{n-k}\\
&-(k-1)(n-k)d^c\om\we\wph^{k-2}\we d\om\we\om^{n-k-1}+(k-1)(n-k)d\om\we\wph^{k-2}\we d^c\om\we\om^{n-k-1}\\
&-(n-k)(n-k-1)\wph^{k-1}\we d\om\we d^c\om\we\om^{n-k-2}+(n-k)\wph^{k-1}\we\dc\om\we\om^{n-k-1}].
\end{align*}
This quantity can be estimated (recall $u$ is uniformly bounded) by
$$C\int_X[\wph^{k-1}\we\om^{n-k+1}+\wph^{k-2}\we\om^{n-k+2}+\wph^{k-3}\we\om^{n-k+3}]$$
(if $k<3$ the terms with negative power of $\wph$ do not appear). Now the proof is finished by induction.
\end{proof}

For $\om$ K\"ahler the capacity is equicontinuous with the Bedford-Taylor capacity denoted in \cite{koj03} or \cite{koj05} by $\caw '$. The latter can also be defined on non-K\"ahler manifolds and is equivalent to its non-K\"ahler counterpart by the decomposition (\ref{de}).
Finally, $\caw $ and $\caw '$ are equicontinuous in the Hermitian case
by the same proof as in \cite{koj03} except that in each strictly pseudoconvex domain $V_s$ one considers two local K\"ahler forms $\om_{1,s}$ and $\om_{2,s}$ satisfying $\om_{1,s}\leq\om\leq\om_{2,s}$ and works with the potentials of those metrics.

Coupling this fact with the argument from \cite{koj03} (Lemma 4.3) one obtains the following corollary:
\begin{corollary}\label{lpcap}
 Let $p>1$ and $f$\ be a non negative function belonging to $L^p(\om^n)$. Then for some absolute constant $C$ dependent only on $(X,\om)$ and any compact $K\subset X$ one has
$$\int_Kf\om^n\leq C(p,X)||f||_p\caw(K)exp(-C\caw^{-1/n}(K)),$$
with $C(p,X)$ a constant dependent on $p$\ and $(X,\om)$.
\end{corollary}

As yet another consequence of psh-like property of $\om-psh$ functions
one gets the capacity estimate of sublevel sets of those functions.

\begin{proposition}\label{capest}
 Let $u\in\wpsh, sup_X u=0$. Then there exists an independent constant $C$\ such that for any $s>1$ $\caw(\lbr u<-t\rbr)\leq\frac Ct$.
\end{proposition}
\begin{proof}
 We shall use the double covering introduced in Proposition \ref{l1int}. Fix a function $v\in\wpsh, 0\leq v\leq 1$. Then we obtain
\begin{align*}
 &\int_{\lbr u<-t\rbr}\om_v^n\leq\frac1t\int_X- u\om_v^n\leq\frac1t(\sum_{s=1}^N\int_{B_s^1}-u(\om_{2,s}+\dc v)^n)\\
&\leq\frac1t(\sum_{s=1}^N\int_{B_s^1}-(u+\rho_s)(\dc(\rho_s+v))^n).
\end{align*}
Now by the generalized $L^1$ Chern-Levine-Nirenberg inequalities (see, for example \cite{dem}, Proposition 3.11) applied to each pair $B_s^1\subset\subset B_s^2$ one  obtains that the last quantity can be estimated by
$$\frac1t\sum_{s=1}^NC_{B_s^1, B_s^2}|| u+\rho_s||_{L^1(B_s^2)}||\rho_s+v||_{L^{\infty}(B_s^2)}\leq\frac1tmax_s\lbr C_{B_s^1, B_s^2}\rbr(C_5N\int_X- u\om^n+C)(C+1)^n,$$
where constants $C_{B_s^1, B_s^2}$ depend on the covering, while $C_5$ - only on $(X,\om)$. By Proposition \ref{l1int} this quantity is uniformly bounded and the statement follows.
\end{proof}

We finish this preliminary section with a  lemma which shall be used throughout the note. It follows from the proof of the comparison principle by Bedford and Taylor in \cite{BT1}.
\begin{lemma}\label{bedtay}
 Let $u,\ v$\ be $\wpsh$ functions and $T$ a (positive but non necessarily closed) current of the form $\om_{u_1}\we\cdots\we\om_{u_{n-1}}$ for bounded functions $u_i$ belonging to $\wpsh$. Then
$$\inpp\dc(u- v)\we T\geq\inpp d^c(u-v)\we dT.$$
\end{lemma}

\section{comparison principles}
In this section we establish the comparison principle
in various forms in the non-K\"ahler case. It has the same form as for K\"ahler
forms (comp. \cite{koj03}) under an extra assumption (\ref{necessary}) below.
Otherwise we get some extra terms, but this general form is sufficient for later applications.

Note that since for any bounded $\om$-psh function $u$ we have for a suitable constant $C$ that $u-C<0<u+C$ a necessary condition for such an inequality to hold is the following one
\begin{equation}\label{necessary}
\forall u\in \wpsh\cap L^{\infty}(X)\ \ \int_X(\om+\dc u)^n=\int_X\om^n.
\end{equation}
Below we show that (\ref{necessary}) is also a sufficient condition:
\begin{proposition}
 If (\ref{necessary}) holds, then for any $u,\ v \in \wpsh\cap L^{\infty}(X)$ we have
$$\inpp\wpsn\leq\inpp\wphn.$$
\end{proposition}
\begin{proof}
 It follows from the locality of the Monge-Amp\`ere operator (which is independent of the underlying metric) \cite{BT1} (see also \cite{GZ2}, \cite{EGZ}) that
$$(\om+\dc\max{(u, v)})^n|_{\lbr u> v\rbr}=(\om+\dc u)^n|_{\lbr u> v\rbr}.$$
Repeating the argument from \cite{GZ2} we obtain for any $\epsilon>0$
\begin{align*}
&\int_{\lbr u-\epsilon< v\rbr}\wpsn=\int_{\lbr u-\epsilon< v\rbr}(\om+\dc\max{(u, v)})^n\\
&=\int_X(\om+\dc\max{(u, v)})^n-\int_{\lbr u-\epsilon\geq v\rbr}(\om+\dc\max{(u, v)})^n\\
&\leq\int_X\wphn-\int_{\lbr u-\epsilon>v\rbr}\wphn=\int_{\lbr u-\epsilon\leq v\rbr}\wphn,
\end{align*}
where we have used condition (\ref{necessary}) and the positivity of the measure in passing from the second line to the last one.

Letting $\epsilon\searrow0$ and using monotone convergence one obtains the claimed result.
\end{proof}

The condition used by Guan and Li in their papers \cite{GL} and \cite{GL2}, namely $\dc\om =0,\ d^c \om \we d\om =0,$ implies  (\ref{necessary}), so comparison principle is true in this setting. In fact  it is enough to have only inequalities in this condition as the next proposition shows.
\begin{proposition}
Assume  $$
dd^c \om \geq 0, \ \ d^c\om\we d\om \geq 0.
$$
Then the comparison principle from the preceding proposition holds.
\end{proposition}
\begin{proof}
Let $u,v$ be smooth $\om$-psh. Then the general case will follow by the psh-like property and the argument for psh functions (see e.g. \cite{koj05}).
Set $U=\{ u<v\}$, and define $T\geq 0$ by
$$
\om _u ^n -\om _v ^n = (\om _u -\om _v )\we T .
$$
\begin{fact} For some positive currents $T_1 , T_2 , ... T_{n-1}$ we have
$$
dT=d\omega\we T_1, \ \ dT_1 =d\om \we T_2, ... , \ \ dT_{n-2}=d\om\we T_{n-1}.
$$
\end{fact}
\begin{proof} It is enough to observe that
$$
d[\om _u ^p \we \om _v ^q ]= pd\om \we \om _u ^{p-1}\we \om _v^q
+qd\om \we \om _u ^{p}\we \om _v^{q-1}.
$$
\end{proof}
By comparing the forms of the same bidegree one easily sees also the following identity.
\begin{fact} $d(u-v)\we d^c \om \we S = -d^c (u-v)\we d\om$ for $S$ positive.
\end{fact}

Let us denote by $A$ the boundary term $\int _{\dd U } d^c (u-v)\we T \geq 0$ (non negativity  of $A$ follows from the proof of the comparison principle in \cite{BT1}). We thus have,
applying the above facts,
$$\aligned
&\int _U \om _u ^n -\om _v ^n = \int _{\dd U } d^c (u-v)\we T
- \int _U  d^c (u-v)\we d\om \we T_1 \\ &= A +\int _U  d^c (u-v)\we d\om \we T_1 = A +\int _U d[(u-v)\we d^c\om \we T_1 ] \\
& -\int _U (u-v) dd^c \om \we T_1 -\int_U (u-v)d^c \om \we dT_1 \\ &\text{ (by Stokes and   } dd^c\om \geq 0) \\
& \geq A- \int _U (u-v)d^c\om \we d\om \we T_2 .
\endaligned
$$
The last integral is nonpositive if $d^c\om \we d\om \we T_2\geq 0$
which follows from our hypothesis.
\end{proof}

 The {\it balanced} metrics, studied extensively (see \cite{TW1} and references therein for more details), are defined by $d(\om^{n-1})=0$. They need not satisfy the condition (\ref{necessary}), except in the case $n=2$. Instead, the reasoning from the first proposition of this section gives us the following comparison principle for the Laplacian in this case:

{\it Let $\om$ be a balanced metric and let $\phi,\ \psi \in \wpsh\cap L^{\infty}(X)$. Then
$$\inpp\wps\we\om^{n-1}\leq\inpp\wph\we\om^{n-1}.$$}

Now we present a weaker form of comparison principle with "error terms" which will be useful in obtaining a priori estimates:
\begin{theorem}\label{main1}
 Let $\om$ be a Hermitian metric on a complex compact manifold $X$ and let $u, v \in \wpsh\cap L^{\infty}(X)$. Then there exists a polynomial $P_n$ of degree $n$, with zeroth coefficient equal to $0$, depending only on the dimension,  such that
$$\inpp\wpsn\leq\inpp\wphn+P_n (B M )\sum_{k=0}^n\inpp\wph^k\we\om^{n-k},$$
where $B$ is defined by (\ref{B})) and $M= sup_{\lbr u< v\rbr}(v-u)$.
\end{theorem}
\begin{proof}
 Note that
\begin{align*}
 &\inpp\wpsn=\inpp\om\we\wps^{n-1}+\inpp\dc v\we\wps^{n-1}\\
&\leq \inpp\om\we\wps^{n-1}+\inpp\dc u \we\wps^{n-1}+\inpp d^c(v-u)\we d(\wps^{n-1}),
\end{align*}
where we have used Lemma \ref{bedtay}.
Again by (\ref{B}) we have
$$\dc(\wps^{n-1}) \leq B[\om ^2\we \wps^{n-2} +\om ^3\we \wps^{n-3}] .
$$
 Thus by Stokes theorem
\begin{align*}
 &\inpp\wpsn\leq\inpp\wph\we\wps^{n-1}-\inpp d(v-u)\we d^c(\wps^{n-1})\\
&\leq \inpp\wph\we\wps^{n-1}+\inpp (v-u)\we \dc(\wps^{n-1})\\
&\leq \inpp\wph\we\wps^{n-1}+sup_{\lbr u< v\rbr}(v-u)B\inpp(\om^2\we\wps^{n-2}+\om^3\we\wps^{n-3}).
\end{align*}
Repeating the above procedure of replacing $\wps$ by $\om$\ and $\wph$ in the end one obtains the statement.
\end{proof}

 In complex dimension $2$  the next proposition yields better inequality.
Recall first a classical notion in Hermitian geometry, the Gauduchon  metric (see \cite{GA}).
 Let $\phi$\ be the unique function such that $inf_X \phi =0$\ and $\eu\om$\ is a Gauduchon metric i.e.
 \begin{equation}\label{gauduchon}
\dc (e^{(n-1)\phi }\om^{n-1})=0.
\end{equation}
\begin{proposition}\label{2compprin}
Let $\om$ be a Hermitian metric on a complex compact manifold $X$ of dimension 2 and let $u,\ v \in \wpsh\cap L^{\infty}(X)$. Let also $\phi$ be defined by the Gauduchon condition with respect to $\om$. Then the following comparison principle holds
$$\inpp2(\eu-1)\wps\we\om+\inpp\wps^2\leq\inpp2(\eu-1)\wph\we\om+\inpp\wph^2.$$
\end{proposition}
\begin{proof}
 Note that
\begin{align*}
 &\inpp(\eu-1)\wps\we\om+\inpp\wps^2=\inpp(\eu\om+\dc v)\we\wps\\
&\leq \inpp(\eu\om+\dc v)\we\wph
\end{align*}
(because $\eu\om+\dc\psi$\ is also $\dc$- closed).

The last term equals
\begin{align*}
&\inpp(\eu-1)\om\we\wph+\inpp\wps\we(\eu\om+\dc u)-\inpp(e^u-1)\om\we\wps\\
&\leq\inpp(\eu-1)\om\we\wph+\inpp\wph\we(\eu\om+\dc u)-\inpp(e^u-1)\om\we\wps\\
&=\inpp2(\eu-1)\om\we\wph+\inpp\wph^2-\inpp(e^u-1)\om\we\wps.
\end{align*}
\end{proof}

 We end this section  proving the comparison principle for the Laplace operator.
\begin{proposition}Let $\om$ be a Hermitian metric and let $u,\ v \in \wpsh\cap L^{\infty}(X)$. Then there exists a constant $C\geq1$\ dependent only on $\om$\ such that
$$\inpp\wps\we\om^{n-1}\leq C\inpp\wph\we\om^{n-1}.$$
\end{proposition}
\begin{proof}

Note that $\inpp\wps\we\om^{n-1}\leq\inpp\wps\we e^{(n-1)\phi}\om^{n-1},$ with $v$\ defined by (\ref{gauduchon}). Since $\int_X\wps\we e^{(n-1)\phi }\om^{n-1}$\ is independent of $v$ (due to the Gauduchon metric assumption) we can repeat the proof of the first proposition in this section to conclude that
$$\inpp\wps\we\om^{n-1}\leq\inpp\wph\we e^{(n-1)\phi}\om^{n-1}\leq e^{(n-1)sup_X \phi}\inpp\wph\we \om^{n-1}.$$
\end{proof}

While the constant $C$ is unsatisfactory for many applications, note that the argument for obtaining $L^{\infty}$\ estimates for equations with $L^p$- right hand side from \cite{koj03} is virtually unaffected by that. So, as a by-product of the above comparison principle we obtain another proof of the $L^{\infty}$\ estimates for the Herimitan Laplacian equation which has the advantage that is independent of any Green type argument and hence independent from cumbersome curvature estimates (strictly speaking, those are hidden in the Gauduchon function).

\section{$L^{\infty}$ estimates}\label{apriori}
In this section we shall prove uniform a priori estimates for the Monge-Amp\`ere equations with $L^p$ right hand side ($p>1$) essentially repeating the proof from \cite{koj03}. Note that whenever the comparison principle holds we have all the needed ingredients and the proof is the same as in \cite{koj03}. The crunch is however that one can prove those a priori estimates using only the weaker form of the comparison principle (Theorem \ref{main1}).

We begin with an auxiliary proposition:
\begin{proposition}\label{bettercapest}
 Let $A>1, t$ and $\epsilon$ be positive constants satisfying $At+\epsilon<\frac1{2B}$, with $B$\ defined by (\ref{B}). Let $u,v <0$ be  bounded $\om$-psh functions such that $1-A<v<0,$ $ sup_X (u-v)=0,$ and $\ inf_X (u-v)=-S$. Then
$$t^n\caw(\lbr u<v-S+\epsilon\rbr)\leq C_n \sum_{k=0}^n\int_{\lbr u<(1-t)v-S+\epsilon+t\rbr}\wph^k\we\om^{n-k},$$
with a constant $C_n$\ dependent only on $n$.
\end{proposition}
\begin{proof}
  Choose any function $w\in\wpsh,\ 0\leq w\leq1$. Note that we have the set inclusions
$$\lbr u<v-S+\epsilon\rbr\subset\lbr u<(1-t)v-S+\epsilon+tw\rbr=:U\subset\lbr u<(1-t)v-S+\epsilon+t\rbr.$$
Then
\begin{equation*}
 t^n\int_{\lbr u<v-S+\epsilon\rbr}\om_w^n\leq \int_{\lbr u<v-S+\epsilon\rbr}\om_{tw}^n\leq\int_{\lbr u<(1-t)v-S+\epsilon+tw\rbr}\om_{tw}^n,
\end{equation*}
where we have used that $t\leq \frac12\leq1$. Since by our assumptions $M=sup_U((1-t)v+t-S+\epsilon-u)<\frac1{2B}$  we conclude from Theorem \ref{main1} the latter quantity can be estimated by $$\int_{U}\wphn+\frac{C_n}{2}
\sum_{k=0}^n\int_U\wph^k\we\om^{n-k}\leq C_n\sum_{k=0}^n
\int_{\lbr u<(1-t)v-S+\epsilon+t\rbr}\wph^k\we\om^{n-k} ,$$
and the result follows.
\end{proof}
Below we state a crucial lemma which   allows to control the mixed Monge-Amp\`ere measures appearing in the above proposition if $v=const.$

\begin{lemma}\label{main2}
 Let $u, t,\ \epsilon$\ be as above (take now $A=1$). Then for every $k\in 1,\cdots,n-1$ we have the estimate
$$\int_{\lbr u<-S+\epsilon+t\rbr}\wph^k\we\om^{n-k}\leq C\int_{\lbr u<-S+\epsilon+t\rbr}(\wphn+\om^n)$$
for some $C$ dependent only on $n$.
\end{lemma}
\begin{proof}
 Set $a_k:=\int_{\lbr u<-S+\epsilon+t\rbr}\wph^k\we\om^{n-k}$. An application of Lemma \ref{bedtay} yields
\begin{align*}
 &a_k\leq a_{k+1}-\int_{\lbr u<-S+\epsilon+t\rbr}d^c(u+S-\epsilon-t)\we d(\wph^k\we\om^{n-k-1})\\
&\leq a_{k+1}+\int_{\lbr u<-S+\epsilon+t\rbr}d(u+S-\epsilon-t)\we d^c(\wph^k\we\om^{n-k-1})\\
&=a_{k+1}+\int_{\lbr u<-S+\epsilon+t\rbr}(\epsilon+t-S- u)\we \dc(\wph^k\we\om^{n-k-1}).
\end{align*}
The last term is bounded by $(\epsilon+t)B(a_k+a_{k-1}+a_{k-2})\leq\frac12(a_k+a_{k-1}+a_{k-2})$. Thus we obtain
\begin{equation}\label{ak}
 a_k\leq 2a_{k+1}+a_{k-1}+a_{k-2}.
\end{equation}
To finish the proof just observe that any sequence $\lbr a_j\rbr_{j=0}^n$ of non negative numbers satisfying (\ref{ak}) also satisfies
$$a_j\leq C(a_n+a_0)$$
with some $C$\ dependent only on $n$.
\end{proof}
Now coupling Theorem \ref{main1} with Lemma \ref{main2} and following the lines of Lemma 2.2 in \cite{koj03} we obtain the following result:
\begin{theorem}\label{main3}
 Let $u$\ be an $\om$-psh function with $sup_X u=0$, while $inf_{X} u=-S$. Suppose for some positive number $\alpha$ and an increasing function $h: \mathbb R_{+}\rightarrow (1,\infty)$ satisfying
$$\int_1^{\infty}(yh^{1/n}(y))^{-1}dy<\infty$$
the following inequality holds
\begin{equation}
 \int_K\wphn\leq F(\caw(K)),\ \ {\rm with}\ F(t)=\frac{\alpha t}{h(t^{-1/n})},\ \alpha >0,
\end{equation}\label{doca}
for any compact set $K$  Then for $D<\frac 1{2B}$ we have
$$D\leq\kappa(\caw(\lbr u<-S+D\rbr)),$$
where
$$\kappa(s):=c(n)\alpha^{1/n}(1+C)[\int_{s^{-1/n}}^{\infty}y^{-1}h^{-1/n}(y)dy+h^{-1/n}(s^{-1/n})].$$
\end{theorem}
As in the K\"ahler case Theorem \ref{main3} together with universal capacity estimates for sublevel sets (Proposition \ref{capest}) yields a priori estimates for $u\in\wpsh$ having Monge-Amp\`ere measure dominated by capacity.

\section{Weak solutions}
Here we prove the existence of continuous solutions of the complex Monge-Amp\`ere equation for nonnegative right hand side which belongs to some Orlicz spaces (including $L^p$, $p>1$)  working under the assumption (\ref{necessary}).

\begin{lemma}\label{unif}
Assume that $\omega$ satisfies (\ref{necessary}) and that $u_j \in PSH(X, \om )\cap C(\bar{X})$
 is a uniformly bounded sequence
converging weakly to $u\in PSH(\om )$.
Suppose further that for $h$  as in Theorem \ref{main3}
$$
\wphj^n =f_j\, \om ^n ,
$$
with $f_j $ satisfying
\begin{equation}
 \int_K f_j \om ^n \leq F(\caw(K)),\ \ {\rm with}\ F(t)=\frac{\alpha t}{h(t^{-1/n})},\ \alpha >0,
\end{equation}
 and such that all $f_j$ belong to the Orlicz space $L^{\psi }(X),$ where $\frac{\psi (x)}{x}$
is increasing to $\infty$ as $x$ goes to $\infty .$
 Then $u_j \to u $ uniformly
in $X.$
\end{lemma}
\begin{proof} Assume $0< u<A-1$. Choose $\de >0$ and $\et =\min (\de , 1/(4A))$ so that
the set $E_j (\de +2A\et ) =\{u_j +\de +2A\et \leq u \}$ is nonempty, whereas
the set $E_j (\de +3A\et ) =\{u_j +\de +3A\et\leq u \}$ is empty.
Denote by $a_j (\de +A\et ) = \caw (E_j (\de +A\et )) $ the  capacity
of the set
$E_j (\de +A\et)$. Let
$v_j$ be a $PSH(X, \om )$ function satisfying $-1<v_j <0$ and
$$
\int_{E_j (\de +A\et ) }(dd^c v_j +\om )^n \geq a_j (\de +A\et )/2 .
$$
Observe that for $V=\{ u_j \leq \et v_j +(1-\et )u-\de \}$
the following inclusions are true
$$
E_j (\de +A\et ) \subset V\subset E_j (\de ).
$$
Applying the comparison principle  we get
$$\aligned
a_j (\de +A\et ) \et ^n &\leq \int _{E_j (\de +A\et ) } (dd^c  (\et v_j +(1-\et )u ) +\om )^n \leq \int _V \wphj^n \\
&\leq \int_{E_j (\de ) }f_j  \om ^n  .
\endaligned
$$
Therefore, using the notation $u_+ :=\max (u,0)$    we have for any $M>0$
$$\aligned
a_j (\de +A\et ) \et  ^{n}\de &\leq \int _{X } (u-u_j )_+ f_j \om ^n  \\
=& \int _{\{ f_j >M \}  } (u-u_j )_+ f_j \om ^n  +
\int _{\{ f_j \leq M\}  } (u-u_j )_+ f_j \om ^n \\
&\leq \max _{X}(u-u_j )_+ \int _{\{ f_j >M \}  }  f_j \om ^n
+M\int _{X } (u-u_j )_+  \om ^n  \\
&\leq \max _{X}(u-u_j )_+ \frac{M}{\psi (M)} \int _{X  } \psi
f_j \om ^n  +M\int _{\Om  } (u-u_j )_+  \om ^n  .
\endaligned\label{last}
$$

For the last inequality we use the assumption that $\frac{\psi (x)}{x}$ is increasing.
 Fix $\ep >0 .$ By the assumptions and the $L^{\infty}$ estimates above the quantities
$$
\max _{X}(u-u_j )_+
\int _{X} \psi f_j \om ^n
$$
are uniformly bounded. Using the assumptions on $\psi $ we can make
$\frac{M}{\psi (M)}$ arbitrarily small by taking $M$ big enough.
We choose $M$ so that the first term on the right hand side of the last estimate
 is less than $\ep /2 $ for any $j$. Since $u_j \to u$ in $L^1 (\Om )$, by the psh-like property, the other
term is less than $\ep /2 $ for $j>j_0 $. Therefore
$$
a_j (\de +A\et ) \leq \ep \et ^{-n } \de ^{-1}\ \ \text{ for  } \ j>j_0 .
$$
Since $E_j (\de +2A\et )$ is nonempty.
Then, applying Theorem 4.3  we obtain
$$
\et \leq \kappa (a_j (\de +A\et )) \leq \kappa (\ep  \et ^{-n}\de ^{-1 } ),\ \ j>j_0 .
$$
Since, by the assumption on $h$ we have $\lim _{s\to 0}\kappa (s)=0$
the last inequality yields a contradiction if we take $\ep $ small
enough. Therefore, using the psh-property and the Hartogs lemma we conclude that
$u_j$ tend to $u$ uniformly.
\end{proof}

\begin{theorem} With $\omega$ satisfying (\ref{necessary}) take
$f $ such that
\begin{equation}
\int_X f \om ^n =\int_X  \om ^n ,\ \
 \int_K f \om ^n \leq F(\caw(K)),\ \ {\rm with}\ F(t)=\frac{\alpha t}{h(t^{-1/n})},\ \alpha >0,
\end{equation} for $h$ as in Theorem \ref{main3}, and such that $f$ belongs to the Orlicz space $L^{\psi }(X),$ where $\frac{\psi (x)}{x}$
is increasing to $\infty$ as $x$ goes to $\infty .$ Then there exists continuous function $u\in PSH(X, \om )$ that fulfils the equation
$$
\wph ^n  =f\, \om ^n .
$$
\end{theorem}
\begin{proof}
It follows from \cite{TW2} \cite{GL2} that the theorem is true for smooth positive $f$.
Thus we approximate $f$ by smooth positive functions $f_j$ in $L^{\psi },$
and obtain solutions $u_j$.
A priori estimates from Section 4 yield uniform bounds for $u_j$. Passing to a subsequence we can therefore assume that $u_j \to u$ in $L^1 (X)$. The last lemma says that the convergence of functions is uniform. The statement now follows by monotone convergence theorem (see Preliminaries).
\end{proof}

\section{A remark on the uniqueness}
In \cite{TW2} the authors have proved that if $u, v$ are smooth $\om$-psh functions and their Monge-Amp\`ere measures satisfy $\wphn=e^{F+b}\om^n,\ \wpsn=e^{F+c}\om^n$ for some smooth function $F$ and some constants $b$ and $c$ then in fact $b=c$ and $u$ and $v$ differ by a constant. This corresponds to the uniqueness of potentials in the Calabi conjecture from the K\"ahler case.

Below we give an alternative proof of the fact that $b=c$ which is in the spirit of pluripotential theory:

Suppose on contrary that
$$\wphn=e^{F+b}\om^n,\ \wpsn=e^{F+c}\om^n$$
for some $u, v$ and without loss of generality assume that $c>b$.

Consider the Hermitian metric $\om+\dc u$. Since by assumptions above it is smooth and strictly positive one finds a unique Gauduchon function $\phi_{u}$, such that $$inf_X\phi_{u}=0,\ \dc(e^{(n-1)\phi_{u}}(\om+\dc u)^{n-1})=0 .$$

Now one can apply the comparison principle for the Laplacian with respect to \newline $e^{\phi_{u}}(\om+\dc u)$ yielding
$$\inpp e^{(n-1)\phi_{u}}(\om+\dc u)^{n-1}\we\wps\leq\inpp e^{(n-1)\phi_{u}}\wphn.$$
Exchanging now $v$\ with $v+C$\ (which does not affect the reasoning above) for big enough $C$ one obtains
\begin{equation}\label{uniqma1}
\int_Xe^{(n-1)\phi_{u}}(\om+\dc u)^{n-1}\we\wps\leq\int_X e^{(n-1)\phi_{u}}\wphn.
\end{equation}
However the left hand side can be estimated from below using (pointwise) the inequality for mixed Monge-Amp\`ere measures (see \cite{Di1} for a discussion on the topic):
$$\int_Xe^{(n-1)\phi_{u}}(\om+\dc u)^{n-1}\we\wps\geq \int_Xe^{(n-1)\phi_{u}+\frac{(c-b)}n}\wphn.$$
Coupling the above estimates one obtains
$$0<e^{\frac{(c-b)}n}-1\leq 0,$$
a contradiction.
\begin{observation}\label{obs}
Applying the Calabi argument directly, as in \cite{TW1} one obtains that $v-u$ would be a subharmonic function with respect to some elliptic operator (dependent on $v, u$). Thus it might be possible to conclude the result directly in this manner. Note however that Calabi's argument depends heavily on the smoothness of $u$ and $v$ and even in the K\"ahler case extending uniqueness to non smooth solutions is a highly nontrivial matter (we refer to \cite{Di2} for a discussion and to \cite{Bl} and \cite{Di2} for the proofs of uniqueness of bounded and general functions respectively). In our approach the smoothness is needed only to produce Gauduchon function, which perhaps will be easier to generalize in the non-smooth setting.
\end{observation}
\bigskip

S\l awomir Dinew:
Institute of Mathematics,
Jagiellonian University,
ul. \L ojasiewicza 6,
30-348 Krak\'ow,
Poland
{\tt slawomir.dinew@im.uj.edu.pl}\\
Current adress:
Erwin Schr\"odinger Institute for mathematical physics, Boltzmanngasse 8, A2010 Vienna, Austria

\bigskip
S\l awomir Ko\l odziej:
Institute of Mathematics,
Jagiellonian University,
ul. \L ojasie\-wi\-cza 6, 30-348 Krak\'ow, Poland
{\tt slawomir.kolodziej@im.uj.edu.pl}

\begin{thebibliography}{EGZ}
\bibitem{Bl} Z. B\l ocki, {\it Uniqueness and stability for the Monge-Amp\`ere equation on compact K\"ahler manifolds,} Indiana Univ. Math. J. {\bf 52} (2003), no. 6 1697-1701.
\bibitem{Bl1} Z. B\l ocki, {\it On uniform estimate in the Calabi-Yau theorem,} Proc SCV2004, Beijing, Science in China Series A Mathematics {\bf 48}, 92005), 244-247.
\bibitem{BT1} E. Bedford and B. A. Taylor, {\it The Dirichlet problem for a complex Monge-Amp\`ere operator,} Invent. Math. {\bf37} (1976), 1-44.
\bibitem{BT2} E. Bedford and B. A. Taylor, {\it A new capacity for plurisubharmonic functions,} Acta Math. {\bf149} (1982), 1-40.
\bibitem{Ch} P. Cherrier, {\it \'Equations de Monge-Amp\`ere sur les vari\'et\'es Hermitiennes compoactes,} Bull. Sc. Math {\bf 111} (2) (1987), 343-385.
\bibitem{CH1} P. Cherrier, {\it Le probl\`eme de Dirichlet pour des \'equations de Monge-Amp\`ere complexes modifi\'ees,} J. Funct. Anal. {\bf156} (1998), 208-251.
\bibitem{CH2} P. Cherrier, {\it Le probl\`eme de Dirichlet pour des \'equations de Monge-Amp\`ere  en m\'etrique hermitienne,} Bull. Sci. Math. {\bf123} (1999), 577-597.
\bibitem{dem} J.-P. Demailly, {\it Complex analytic and differential geometry,} self published e-book.
\bibitem{Di1} S. Dinew, {\it An inequality for mixed Monge-Amp\`ere measures,} Math. Zeit. {\bf 262} (2009), 1-15.
\bibitem{Di2} S. Dinew, {\it Uniqueness in $\mathcal E(X,\omega)$,} J. Funct. Anal. {\bf 256} (7) (2009), 2113-2122.
\bibitem{DNS} T. C. Dinh, V. A. Nguyen and N. Sibony, {\it Exponential estimates for plurisubharmonic functions and stochastic dynamics,} preprint arXiv 0801.1983.
\bibitem{EGZ} P.Eyssidieux, V.Guedj and A.Zeriahi, {\it Singular K\"ahler-Einstein matrics,} Journal AMS {\bf 22} (2009), 607-639.
\bibitem{FS} J. E. Fornaess and N. Sibony, {\it Harmonic currents of finite energy and laminations,} GAFA {\bf 15} (2005), 962-1003.
\bibitem{FY} J.-X. Fu and S.-T. Yau, {\it The theory of superstring with flux on
    non-K\"ahler manifolds and  the complex Monge-Amp\`ere equation,} J. Diff. Geom. {\bf98} (2008), 369-428.
\bibitem{GA} P. Gauduchon, {\it La th\'eor\`eme de l'excentricit\'e nulle,} C. R. Acad. Sci. Paris {\bf 285} (1977), 387-390.
\bibitem{GL} B. Guan and Q. Li, {\it Complex Monge-Amp\`ere equations on Hermitian manifolds,} preprint arXiv: 0906.3548v1.
\bibitem{GL2} B. Guan and Q. Li, {\it Complex Monge-Amp\`ere equations and totally real submanifolds,} preprint arXiv: 0910.1851.
\bibitem{GZ1} V.Guedj and A.Zeriahi, {\it Intrinsic capacities on compact K\"ahler manifolds} J. Geom. Anal. {\bf 15} (2005), 607-639.
\bibitem{GZ2} V.Guedj and A.Zeriahi, {\it The weighted Monge-Amp\`ere energy of quasiplurisubharmonic functions,} J. Funct. Anal. {\bf 250} (2) (2007), 442-482.
\bibitem{H1} A. Hanani, {\it \'Equations du type de Monge-Amp\`ere sur les vari\'et\'es hermitiennes compactes,} J. Funct. Anal. {\bf137} (1996), 49-75.
\bibitem{H2} A. Hanani, {\it Une g\'en\'eralisation de l'\'equation de Monge-Amp\`ere sur les vari\'et\'es hermitiennes compactes,} Bull. Sci. Math. {\bf120} (1996), 215–252.
\bibitem{koj98} S. Ko\l odziej, {\it The complex Monge-Amp\`ere equation,} Acta Math. {\bf 180} (1998), 69-117.
\bibitem{koj03} S. Ko\l odziej, {\it The Monge-Amp\`ere equation on compact K\"ahler manifolds,} Indiana Univ. Math. J. {\bf52} (2003), no. 3, 667-686.
\bibitem{koj05} S. Ko\l odziej, {\it The complex Monge-Amp\`ere equation and pluripotential theory,} Memoirs Amer. Math. Soc. {\bf 178} (2005), pp. 64.
\bibitem{koj08} S. Ko\l odziej, {\it H\"older continuity of solutions to the complex Monge-Amp\`ere equation with the right hand side in $L^p$.
    The case of compact K\"ahler manifolds,} Math. Ann. {\bf342} (2008), 379-386.
\bibitem{StT} J. Streets and G. Tian, {\it A parabolic flow of pluriclosed metrics,} preprint 0903.4418.
\bibitem{TW1} V. Tosatti and B. Weinkove, {\it Estimates for the complex Monge-Amp\`ere equation on Hermitian and balanced manifolds,} preprint arXiv: 0909.4496v2.
\bibitem{TW2} V. Tosatti and B. Weinkove, {\it The complex Monge-Amp\`ere equation on compact Hermitian manifolds,} preprint arXiv: 0910.1390v1.
\bibitem{TWY} V. Tosatti, B. Weinkove and S.-T.Yau, {Taming symplectic forms and the Calabi-Yau equation,} Proc. London Math. Soc. {\bf97} (2) (2008), 401-424.
\bibitem{Y} S.-T. Yau, {On the Ricci curvature of a compact K\"ahler manifold and the complex Monge-Amp\`ere equation,} Comm. Pure Appl. Math. {\bf 31} (30 (19780, 339-411.
\end{thebibliography}
\end{document}